\documentclass[a4paper]{amsart}
\usepackage[margin=1.35in]{geometry}
\usepackage{tikz,tikz-cd}
\input{commands_new.tex}
\usepackage[colorlinks]{hyperref}

\def\hom#1{\Hom_{#1}}

\opr{Rep}{\mathsf{Rep}}
\opr{dRep}{\overline{\mathsf{Rep}}}
\opr{bGa}{\bf\Ga}
\opr{bHom}{\bf Hom}
\opr{fCoh}{Coh_{\mathrm f}}

\def\sir#1{\si_*(#1)}

\opr{Rep}{\overline{Rep}}
\opr{dRep}{Rep}
\excludeversion{notes}
\excludeversion{old}

\begin{document}
\title{Quiver representations in abelian categories}
\author{Sergey Mozgovoy}

\address{School of Mathematics, Trinity College Dublin, Ireland\newline\indent
Hamilton Mathematics Institute, Ireland}

\email{mozgovoy@maths.tcd.ie}

\begin{abstract}
We introduce the notion of (twisted) quiver representations in abelian categories and study the category of such representations.
We construct standard resolutions and coresolutions of quiver representations and study basic homological properties of the category of representations. 
These results are applied in the case of parabolic vector bundles and framed coherent sheaves.
\end{abstract}

\maketitle

\section{Introduction}
Let $\Phi:\cA\to\cB$ be an additive functor between abelian categories.
Its mapping cylinder $\cC_\Phi$
is the category with objects consisting of triples $(A,B,\phi)$, where $$A\in\Ob(\cA),\qquad B\in\Ob(\cB),\qquad \phi\in\Hom_\cB(B,\Phi A).$$
This category is abelian if \Phi is left exact.
For example, let $\cA=\Coh Z$ be the category of coherent sheaves on an algebraic variety~$Z$ over a field \bk,
$\cB=\Vect$ be the category of vector spaces over \bk
and $\Phi=\Ga(Z,-)$ be the section functor which is left exact.
An object of the mapping cylinder is a triple $(E,V,\vi)$, where 
$$E\in\Coh Z,\qquad V\in\Vect,\qquad \vi:V\to\Ga(Z,E).$$
The category $\fCoh Z=\cC_\Phi$
is the category of framed coherent sheaves over $Z$~\cite{mozgovoy_wall-crossing}.
One can show \cite{mozgovoy_intersection} that if $\Phi:\cA\to\Vect$ is exact and \cA is hereditary,
then $\cC_\Phi$ is also hereditary.

The goal of this note is to develop some basic homological algebra of mapping cylinders and, more generally, of categories of (twisted) quiver representations in abelian categories.
Let $Q=(Q_0,Q_1,s,t)$ be a quiver and $\Phi$ be a $Q^\op$-diagram of abelian categories (see \S\ref{sec:quivers})
consisting of abelian categories $\Phi_i$ for all $i\in Q_0$ and left exact functors $a^*=\Phi_a:\Phi_j\to\Phi_i$ for all arrows $a:i\to j$ in $Q$.
Define the category $\dRep(\Phi)$ of (twisted) $Q$-representations to have objects $X$ consisting of the data: $X_i\in\Ob(\Phi_i)$ for all $i\in Q_0$ and $X_a:X_i\to a^*X_j$ for all arrows $a:i\to j$ in $Q$.
We will see in Theorem \ref{th:abelian} that $\dRep(\Phi)$ is an  abelian category.

A different way of thinking about $\dRep(\Phi)$ is to consider \Phi as a functor $\cP^\op\to\Cat$, where~$\cP$ is the path category of $Q$ (see \S\ref{sec:quivers}) and $\Cat$ is the category of small categories.
It induces a fibered category $\cF\to\cP$.
Then the category $\dRep(\Phi)$ can be identified with the  category of sections $\bGa\rbr{\cF/\cP}$ (see Remark \ref{fibered cat}).

\begin{notes}
We can consider the category $\cA=\prod_i\cA_i$ and induced functors $\Phi_a:\cA\to\cA_j\to\cA_i\to\cA$.
\end{notes}

Apart from the usual $Q$-representations in the category of vector spaces, an important example of $Q$-representations was studied by Gothen and King~\cite{gothen_homological}, where all categories $\Phi_i$ are equal to the category $\Sh(Z,\cO_Z)$ of $\cO_Z$-modules over a ringed space $(Z,\cO_Z)$ and the functors $\Phi_a$ are tensor products with locally free sheaves.
This covers a lot of interesting situations, like Higgs bundles, Bradlow pairs and chains of vector bundles (see e.g.\ \cite{garcia-prada_motives}),
although not the category of framed objects introduced above or the category of parabolic bundles.
The case of a one loop quiver and an auto-equivalence was studied in \cite{mellit_poincarea}.

In this paper we describe homological properties of the category $\dRep(\Phi)$ in the spirit of~\cite{gothen_homological}.
In Theorems \ref{standard} and \ref{th:costand} we construct standard resolutions and coresolutions in $\dRep(\Phi)$ which are analogues of standard resolutions for usual quiver representations.
In Theorem~\ref{th:injectives} we prove that $\dRep(\Phi)$ has enough injective objects.
The main results of the paper are

\begin{theorem}
Let $\Phi$ be a $Q^\op$-diagram of Grothendieck categories
such that the functor $\Phi_a:\Phi_j\to\Phi_i$ has an exact left adjoint functor $\Psi_a:\Phi_i\to\Phi_j$ for every arrow $a:i\to j$ in $Q$.
Then, for any $X,Y\in\dRep(\Phi)$,
there is a long exact sequence
\begin{multline*}
0\to\Hom_{\dRep(\Phi)}(X,Y)\to\bop_{i}\Hom_{\Phi_i}(X_i,Y_i)
\to\bop_{a:i\to j}\Hom_{\Phi_j}(\Psi_aX_i,Y_j)\\
\to\Ext^1_{\dRep(\Phi)}(X,Y)
\to\bop_{i}\Ext^1_{\Phi_i}(X_i,Y_i)
\to\bop_{a:i\to j}\Ext^1_{\Phi_j}(\Psi_aX_i,Y_j)\to\dots
\end{multline*}
\end{theorem}

\begin{theorem}
Let $Q$ be an acyclic quiver and $\Phi$ be a $Q^\op$-diagram of Grothendieck categories such that the functor $\Phi_a:\Phi_j\to\Phi_i$ is exact for every arrow $a:i\to j$ in $Q$.
Then, for any $X,Y\in\dRep(\Phi)$,
there is a long exact sequence
\begin{multline*}
0\to\Hom_{\dRep(\Phi)}(X,Y)\to\bop_{i}\Hom_{\Phi_i}(X_i,Y_i)
\to\bop_{a:i\to j}\Hom_{\Phi_i}(X_i,\Phi_aY_j)\\
\to\Ext^1_{\dRep(\Phi)}(X,Y)
\to\bop_{i}\Ext^1_{\Phi_i}(X_i,Y_i)
\to\bop_{a:i\to j}\Ext^1_{\Phi_i}(X_i,\Phi_aY_j)\to\dots
\end{multline*}
\end{theorem}

Note that $\Ext^k_{\Phi_i}(X_i,\Phi_aY_j)\not\iso\Ext^k_{\Phi_j}(\Psi_aX_i,Y_j)$ in general,
although these groups are isomorphic if both $\Phi_a$ and $\Psi_a$ are exact (see Remark \ref{rem:both exact}).
Also note that even if we start with a quiver diagram of small abelian categories (like the category $\Coh Z$ for an algebraic variety~$Z$) which can not be Grothendieck as they don't have small colimits, we can consider their \Ind-categories which are Grothendieck (in the case of $\Coh Z$ we obtain the category $\Qcoh Z$ of quasi-coherent sheaves) and extend our diagram to a diagram of Grothendieck categories (see \S\ref{sec:results}).
Then we can apply the above results.

\medskip
The paper is organized as follows.
In \S\ref{sec:diag reps} we introduce diagrams of categories and define their categories of representations.
We prove that the category of representations is abelian under appropriate assumptions and we show that the forgetful functor sending a representation to one of its components has left and right adjoint functors.
In \S\ref{sec:quivers} we restrict our attention to quiver diagrams and we construct standard resolutions and coresolutions of representations.
In \S\ref{sec:results} we prove our main results about long exact sequences for quiver representations. 
In \S\ref{sec:app} we discuss several applications of these results, in particular, in the case of parabolic vector bundles and framed coherent sheaves.

\section{Diagram representations}
\label{sec:diag reps}
Let $\Cat$ be the category of small categories.
Given a small category \cI,
define an \cI-diagram (or just a diagram) of categories to be a 
functor 
$\Psi:\cI\to\Cat$.
For every object $i\in\Ob(\cI)$ let $\Psi_i=\Psi(i)$ and 
for every morphism
$a\in \cI(i,j)=\Hom_\cI(i,j)$ let 
$a_*=\Psi_a=\Psi(a):\Psi_i\to \Psi_j$.
We will assume that the categories $\Psi_i$ are abelian and the functors $\Psi_a$ are additive unless otherwise stated.
We will say that a diagram \Psi is exact (resp.\ left exact, right exact) if the functors $\Psi_a:\Psi_i\to\Psi_j$ are exact (resp.\ left exact, right exact) for all $a\in\cI(i,j)$.

\begin{definition}
Given a diagram $\Psi:\cI\to\Cat$,
define the category $\Rep(\Psi)$ of 
representations (also called a projective $2$-limit of $\Psi$ in a different context \cite{dyckerhoff_highera})
to have objects $X$ consisting of the data:
\begin{enumerate}
\item An object $X_i\in \Psi_i$ for all $i\in \Ob(\cI)$.
\item A morphism $X_a:a_*(X_i)\to X_j$ in $\Psi_i$ for every 
morphism $a:i\to j$ in $\cI$.
\end{enumerate}
subject to the condition:
\begin{enumerate}
\item[(3)] Given morphisms $a:i\to j$, $b:j\to k$ in \cI, we 
have $X_{ba}=X_b\circ b_*(X_a)$.
\end{enumerate}
A morphism $f:X\to Y$ between two objects in $\Rep(\Psi)$ is a tuple 
$f=(f_i)_{i\in \Ob(\cI)}$ of morphisms $f_i:X_i\to Y_i$ in 
$\Psi_i$ such that for every $a\in\cI(i,j)$, the following 
diagram commutes
$$\begin{tikzcd}
a_*(X_i)\dar[swap]{a_*(f_i)}\rar{X_a}& X_j\dar{f_j}\\
a_*(Y_i)\rar{Y_a}&Y_j
\end{tikzcd}$$
\end{definition}


\begin{example}
Let $\cI$ be the category with
$\Ob(\cI)=\set{1,2}$ and
$$\cI(1,1)=\set{\Id_1},\quad \cI(2,2)=\set{\Id_2},\quad
\cI(1,2)=\set{a},\quad \cI(2,1)=\es.$$
Given an algebraic variety $Z$, let 
$$\Psi_1=\Vect,\qquad \Psi_2=\Coh Z,\qquad \Psi_a:\Vect\to\Coh Z,\ V\mto V\ts\cO_Z.$$
Then the category $\Rep(\Psi)$ consists of triples $(V,F,s)$
with $V\in\Vect$, $F\in\Coh Z$ and $s:V\ts\cO_Z\to F$.
Note that $\Psi_a$ has a right adjoint functor $\Phi_a:\Coh Z\to\Vect$, $F\mto\Ga(Z,F)$.
We can
consider $s$ as a morphism $V\to\Phi_a(F)=\Ga(Z,F)$ in \Vect.
\end{example}

\begin{definition}
Given a diagram $\Phi:\cI^\op\to\Cat$, define the category
$\dRep(\Phi)$ to have objects~$X$ consisting of the data:
\begin{enumerate}
\item An object $X_i\in \Phi_i$ for all $i\in \Ob(\cI)$.
\item A morphism $X_a:X_i\to a^*X_j$ in $\Phi_i$ for every 
morphism $a:i\to j$ in $\cI$, where $a^*=\Phi_a$.
\end{enumerate}
subject to the condition:
\begin{enumerate}
\item[(3)] Given morphisms $a:i\to j$, $b:j\to k$ in \cI, we 
have $X_{ba}=a^*(X_b)\circ X_a$.
\end{enumerate}
\end{definition}

\begin{definition}
We say that a diagram $\Psi:\cI\to\Cat$ has a right adjoint diagram $\Phi:\cI^\op\to\Cat$ if $\Phi_i=\Psi_i$ for all $i\in\Ob(\cI)$ and $\Psi_a:\Psi_i\to\Psi_j$ is right adjoint
to $\Phi_a:\Phi_j\to\Phi_i$ for all $a\in\cI(i,j)$.
\end{definition}

\begin{remark}
\label{rem:prod and coprod}
Let $\Psi:\cI\to\Cat$ have a right adjoint diagram $\Phi:\cI^\op\to\Cat$.
Then $\Rep(\Psi)$ is equivalent to $\dRep(\Phi)$.
Moreover, the diagram $\Psi$ is right exact, the diagram~$\Phi$ is left exact,
$\Psi_a:\Psi_i\to\Psi_j$ preserves coproducts and $\Phi_a:\Phi_j\to\Phi_a$ preserves products for all $a\in\cI(i,j)$.
\end{remark}

\begin{remark}
Given a diagram $\Psi:\cI\to\Cat$,
define the opposite diagram $\Psi^\op:\cI\to\Cat$
with $\Psi^\op_i$ being the opposite category of $\Psi_i$
and with the opposite functors $\Psi_a^\op:\Psi_i^\op\to\Psi_j^\op$ for $a\in\cI(i,j)$.
There is a canonical equivalence of categories
\eq{\dRep(\Psi^\op)\iso\Rep(\Psi)^\op.
\label{duality}}
This duality implies that the statements about the category $\dRep(\Psi^\op)$ can be translated to the statements about the category $\Rep(\Psi)$ and vice versa.
\end{remark}

\begin{remark}
\label{fibered cat}
Given a diagram $\Phi:\cI^\op\to\Cat$, 
the category $\dRep(\Phi)$ can be interpreted as the category of sections of the associated fibered category
$p:\cF\to\cI$.
More precisely, the category $\cF=\int\Phi$, called the Grothendieck construction for $\Phi$ \cite[\S VI.8]{SGA1}, has objects
$$\Ob(\cF)=\coprod_{i\in\Ob(\cI)}\Ob(\Phi_i)$$
and, for $X_i\in\Ob(\Phi_i)$ and $X_j\in\Ob(\Phi_j)$, morphisms
$$\Hom_\cF(X_i,X_j)=\coprod_{a\in\Hom_\cI(i,j)}\Hom_{\Phi_i}(X_i,a^*X_j).$$
The functor $p:\cF\to\cI$ is defined by
$$\Ob(\Phi_i)\ni X_i\mto i\in\Ob(\cI),\qquad \Hom_{\Phi_i}(X_i,a^*X_j)\ni X_a\mto a\in \Hom_\cI(i,j).$$
Define the category of sections $\bGa(\cF/\cI)=\bHom_\cI(\cI,\cF)$
whose objects are functors $X:\cI\to\cF$ such that $pX=\Id_\cI$ and morphisms from $X:\cI\to\cF$ to $Y:\cI\to\cF$ are morphisms of functors $f:X\to Y$
such that $pf:pX\to pY$ is an identity morphism of functors.
The last condition means that the component $f_i\in\Hom_\cF(X_i,Y_i)$ of $f$ is contained in $\Hom_{\Phi_i}(X_i,Y_i)$ for every $i\in\Ob(\cI)$.
One can show that the category $\bGa(\cF/\cI)$ is equivalent to the category of representations $\dRep(\Phi)$.
\end{remark}

\begin{theorem}\label{th:abelian}
Let $\Phi:\cI^\op\to\Cat$ be a left exact diagram.
Then the category of representations $\dRep(\Phi)$ is 
abelian.
A sequence of representations $X\to Y\to Z$ is exact in $\dRep(\Phi)$ 
if and only if the sequence
$X_i\to Y_i\to Z_i$
is exact in $\Phi_i$
for all $i\in \Ob(\cI)$.
\end{theorem}
\begin{proof}
We just have to show that there exist kernels and cokernels 
in $\cR=\dRep(\Phi)$ which are
defined componentwise.
Let $f:X\to Y$ be a morphism in $\cR$.
We construct the kernel of $f$ as follows.
For every $i\in \Ob(\cI)$, let $h_i:K_i\to X_i$ be the kernel 
of 
$f_i:X_i\to Y_i$.
For every arrow $a:i\to j$, consider the diagram
$$\begin{tikzcd}
K_i\dar[dashed,"K_a"]\rar["h_i"]&
X_i\rar["f_i"]\dar["X_a"]&
Y_i\dar["Y_a"]
\\
a^*K_j\rar["a^*(h_j)"]&
a^*X_j\rar["a^*(f_j)"]&
a^*Y_j
\end{tikzcd}$$
Morphism $a^*(h_j)$ is the kernel of $a^*(f_j)$ as 
$a^*$ is left exact
by our assumption.
The composition $a^*(f_j)X_ah_i$ is zero, hence there 
exists a unique dashed
arrow $K_a$ making the left square commutative.
In this way we obtain an object $K\in\cR$ and a morphism 
$h:K\to X$ such that
$fh=0$.
Let us show that $h$ is the kernel of $f$.
Let $g:Z\to X$ be a morphism in \cR such that $fg=0$.
Then every $g_i:Z_i\to X_i$ can be uniquely factored as
$$\begin{tikzcd}
Z_i\rar["s_i"]\ar[rr,bend left,"g_i"]&K_i\rar["h_i"]&X_i.
\end{tikzcd}$$
To see that $s=(s_i)_{i\in \Ob(\cI)}$ defines a morphism 
$s:Z\to 
K$ such that $g=hs$,
we have to verify commutativity of the left square in the 
diagram
$$\begin{tikzcd}
Z_i
\rar["s_i"]\dar["Z_a"]&
K_i\dar["K_a"]\rar["h_i"]&
X_i\dar["X_a"]
\\
a^*Z_j
\rar["a^*(s_j)"]&
a^*K_j\rar["a^*(h_j)"]&
a^*X_j
\end{tikzcd}$$
This commutativity follows from the fact that
$$a^*(h_j)K_as_i=X_ah_is_i
=a^*(h_j)a^*(s_j)Z_a$$
and that $a^*(h_j)$ is a monomorphism (as $a^*$ is 
left exact).
This proves that $h:K\to X$ is the kernel of $f:X\to Y$.

Construction of the cokernel is similar, with the only 
difference that one does
not require any exactness properties of $a^*$.
\end{proof}

Applying duality \eqref{duality} we obtain

\begin{corollary}
Let $\Psi:\cI\to\Cat$ be a right exact diagram.
Then the category of representations $\Rep(\Psi)$ is abelian.
\end{corollary}

\begin{remark}
In order to provide a motivation of the
construction in the next theorem, let us consider the case of quiver representations in vector spaces.
Let $A=\bk Q$ be the path algebra of a quiver $Q$ over a field $\bk$ \cite{crawley-boevey_lectures}.
One can identify $Q$-representations over~\bk with $A$-modules.
For any vertex $i\in Q_0$, there is an idempotent $e_i\in A$ corresponding to the trivial path at $i$.
The projective $A$-module $Ae_i$ has a basis consisting of paths that start at~$i$.
Given a vector space $V$ and a $Q$-representation $X$, the $Q$-representation $Ae_i\ts V$ satisfies
$$\hom A(Ae_i\ts V,X)\iso \hom\bk(V,X_i),$$
where $X_i=e_iX\in\Vect\bk$.
This implies that the forgetful functor $\Mod A\to\Vect\bk$, $X\mto X_i$ has a left adjoint $V\mto Ae_i\ts V$.
\end{remark}

\begin{theorem}
\label{th:adj}
Let $\Psi:\cI\to\Cat$ be a diagram such that the categories $\Psi_i$ have coproducts and the 
functors $\Psi_a:\Psi_i\to\Psi_j$ preserve them for all 
$a:i\to j$ in \cI.
Then, for every $i\in\cI$, the forgetful functor
$$\si^{*}:\Rep(\Psi)\to\Psi_i,\qquad X\mto X_i$$
has a left adjoint functor
$\si_!:\Psi_i\to\Rep(\Psi)$.
\end{theorem}
\begin{proof}
\def\lb#1{Y}
Given an object $M\in\Psi_i$, we define an object 
$\lb M=\si_! (M)\in\cR=\Rep(\Psi)$ as follows.
For any $j\in \cI$, define
$$\lb M_j=\si_!(M)_j=\bop_{a\in\cI(i,j)}a_*M\in\Psi_j.$$
Let us construct the map $\lb M_b:b_*Y_j\to Y_k$ for any $b\in\cI(j,k)$.
For any morphisms $a\in\cI(i,j)$
consider the canonical embedding
$$b_*a_* M\iso(ba)_*M\emb \lb M_k.$$
These maps induce
$$\lb M_b:b_*\lb M_j
=b_*\rbr{\bop_{a\in\cI(i,j)}a_*M}
\iso\bop_{a\in\cI(i,j)}b_*a_*M
\to\lb M_k.$$
In this way we obtain an object $\lb M\in\cR$.
For any $X\in\cR$ there is a natural 
isomorphism
\begin{equation}
\label{eq:left adj}
\hom\cR(\lb M,X)\iso \hom{\Psi_i}(M,X_i)
\end{equation}
Given $f\in \hom\cR(\lb M,X)$, we define
$g\in \hom{\Psi_i}(M,X_i)$ to be the composition 
$M\emb\lb M_i\xto {f_i}X_i$
of the map $f_i$ and an embedding corresponding to the identity $\Id\in\cI(i,i)$.
Conversely, given $g\in\Hom_{\Psi_i}(M,X_i)$, for every $j\in\cI$ we construct
$f_j:\lb M_j=\bop_{a\in\cI(i,j)}a_*M\to X_j$ such that
the component $a_*M\to X_j$ is given by
$a_*M\xto{a_*(g)}a_*X_i\xto{X_a}X_j$.
\end{proof}

Applying duality \eqref{duality} we obtain

\begin{theorem}
\label{th:right adj}
Let $\Phi:\cI^\op\to\Cat$ be a diagram such that the categories $\Phi_i$ have products and the 
functors $\Phi_a:\Phi_j\to\Phi_i$ preserve them for all 
$a:i\to j$ in \cI.
Then for every $i\in\cI$ the forgetful functor
$$\si^*:\dRep(\Phi)\to\Phi_i,\qquad X\mto X_i$$
has a right adjoint functor
$\si_{*}=\Phi_i\to\dRep(\Phi)$.
Explicitly, given $M\in\Phi_i$, the object $\sir M$ is defined by
$$\sir M_j=\prod_{a\in\cI(j,i)}\Phi_aM\in\Phi_j.$$
\end{theorem}

\begin{corollary}
\label{cor:inj and proj}
Under the above conditions
\begin{enumerate}
\item 
If $M\in\Psi_i$ is projective, then $\si_!M\in\Rep(\Psi)$ is also projective.
\item
If $M\in\Phi_i$ is injective, then $\si_*M\in\dRep(\Phi)$ is also injective.
\end{enumerate}
\end{corollary}
\begin{proof}
Let us prove the first statement.
If $M\in\Psi_i$ is projective,
then
$$\Hom_{\Rep(\Psi)}(\si_!M,-)\iso\Hom_{\Psi_i}(M,\si^*(-))$$
is an exact functor as a composition of two exact functors $\si^*$ and $\Hom_{\Psi_i}(M,-)$.
This implies that $\si_!M$ is projective.
\end{proof}

\section{Standard resolutions and coresolutions}
\label{sec:quivers}
Recall that a quiver $Q=(Q_0,Q_1,s,t)$ is a directed 
\mbox{(multi-)} graph.
Here $Q_0,\,Q_1$ are sets, called the sets of vertices and 
arrows respectively and $s,t:Q_1\to Q_0$ are maps, called
the source and target maps respectively.
All quivers are assumed to be finite unless otherwise stated.
We denote an arrow $a\in Q_1$ with $i=s(a)$ and 
$j=t(a)$ as $a:i\to j$.
Define the category $\cP=\cP(Q)$ of paths in $Q$ to have the set of objects $Q_0$ and the set of morphisms 
$\cP(i,j)=\Hom_\cP(i,j)$ consisting of all directed paths from $i\in Q_0$ to $j\in Q_0$.
Define a $Q$-diagram of categories to be a functor $\Psi:\cP\to\Cat$.
It is uniquely determined by the categories $\Psi_i$ for $i\in Q_0$ and the functors $a_*=\Psi_a:\Psi_i\to\Psi_j$ for arrows $a:i\to j$ in $Q$.

\begin{definition}
Given a $Q$-diagram $\Psi$, define the category of $Q$-representations (\wrt \Psi) to be the category $\Rep(Q,\Psi)=\Rep(\Psi)$.
Its objects consist of data:
\begin{enumerate}
\item An object $X_i\in \Psi_i$ for all $i\in Q_0$.
\item A morphism $X_a:a_*(X_i)\to X_j$ in $\Psi_j$ for every 
arrow $a:i\to j$ in $Q$.
\end{enumerate}
\end{definition}

Similarly, define a $Q^\op$-diagram to be a diagram $\Phi:\cP(Q)^\op\to\Cat$, which is uniquely determined by the categories $\Phi_i$ for $i\in Q_0$ and the functors $a^*=\Phi_a:\Phi_j\to\Phi_i$ for arrows $a:i\to j$ in $Q$.
In this case we define $\dRep(Q,\Phi)=\dRep(\Phi)$.

\begin{remark}
To give a motivation of the
construction of the next theorem, let us consider the case of quiver representations in vector spaces.
Let $A=\bk Q$ be the path algebra of a quiver $Q$ over a field \bk.
Given a $Q$-representation $X$, there is a short exact sequence of $Q$-representations
\cite[p.7]{crawley-boevey_lectures}
$$0\to \bop_{a:i\to j}Ae_j\ts X_i\to\bop_i Ae_i\ts X_i\to X\to 0,$$
called the standard resolution of $X$.
Our goal is to construct an analogue of this resolution in the category $\Rep(Q,\Psi)$ as well as its dual version, called a coresolution.
\end{remark}

\begin{theorem}
\label{standard}
Let $\Psi:\cP(Q)\to\Cat$ be a diagram such that
the categories $\Psi_i$ have (countable) coproducts and the functors $\Psi_a:\Psi_i\to\Psi_j$ preserve them for all arrows $a:i\to j$ in $Q$.
For any object $X\in\Rep(\Psi)$, there is a short exact sequence
$$0\to\bop_{a:i\to j}\si_!(\Psi_aX_i)\xto\be\bop_{i}\si_! (X_i)\xto\ga X\to0$$
called a standard resolution of $X$, 
where
for any arrow $a:i\to j$ and any path $p:j\to k$,
$$\be:\Psi_p\Psi_aX_i\xto{(\Id,-\Psi_pX_a)} \Psi_{pa}X_i\oplus \Psi_pX_j,
\qquad
\ga:\Psi_pX_j\xto{X_p}X_k.
$$
\end{theorem}

\begin{proof}
It is clear that $\ga\be=0$.
We need to verify exactness of the components for every vertex $k\in Q_0$.
For every $n\ge0$, define the degree $n$ component
$$Z_n=\bop_{\ov{p:i\to k}{l(p)=n}}\Psi_pX_i\in\Psi_k,$$
where $k\in Q_0$ is fixed, $i\in Q_0$ varies, and $l(p)$ is the length of the paths $p$.
We have to verify that the sequence
$$0\to\bop_{n\ge1}Z_n\xto\be\bop_{n\ge0}Z_n\xto\ga Z_0\to0$$
is exact in $\Psi_k$.
The matrices of \be and \ga are of the form
$$
\be=\begin{pmatrix}
*&0&0&0&\dots\\
1&*&0&0&\dots\\
0&1&*&0&\dots\\
0&0&1&*&\dots\\
\hdotsfor5
\end{pmatrix}\qquad
\ga=
\begin{pmatrix}
1&*&*&\dots
\end{pmatrix}
$$
We can assume that $\Psi_k$ is a category of modules over an algebra.
It is clear that \ga is surjective.
If $a\in\Ker\be$ is nontrivial, let $a_n\in Z_n$ be its nonzero component of maximal degree.
Then the restriction of $\be(a)$ to $Z_n$ is $a_n\ne0$. A contradiction.
Finally, given $a\in\bop_{n\ge0}Z_n$, let $\deg a=\max\sets{n\ge0}{a_n\ne0}$.
Choose $0\ne a\in\Ker\ga\ms\Im\be$ with the minimal possible $n=\deg a$.
If $n\ge1$, then $a'=a-\be(a_n)$ has degree $<n$ and we still have $a'\in\Ker\ga\ms\Im\be$ as $\Im\be\sbs\Ker\ga$.
A contradiction. Therefore $a$ has degree zero and $\ga(a)=a=0$. A contradiction.
\end{proof}

Taking the opposite categories we obtain a dual version of the above theorem.

\begin{theorem}
\label{th:costand}
Let $\Phi:\cP(Q)^\op\to\Cat$ be a diagram such that
the categories $\Phi_i$ have (countable) products and the functors $\Phi_a:\Phi_j\to\Phi_i$ preserve them for all arrows $a:i\to j$ in $Q$.
For any object $X\in\dRep(\Phi)$, there is a short exact sequence
$$0\to X\xto\ga 
\bop_{i}\sir {X_i}\xto\be
\bop_{a:j\to i}\sir{\Phi_aX_i}
\to0$$
called a standard coresolution of $X$,
where
for any arrow $a:j\to i$ and any path $p:k\to j$,
$$\be: \Phi_{ap}X_i\oplus \Phi_pX_j
\xto{(\Id,-\Phi_pX_a)}\Phi_p\Phi_aX_i,
\qquad
\ga:X_k\xto{X_p}\Phi_pX_j.
$$
\end{theorem}

\section{Homological properties}
\label{sec:results}
Let $Q$ be a quiver and \Psi be a $Q$-diagram, that is, a collection of categories $(\Psi_i)_{i\in Q_0}$ and functors $a_*=\Psi_a:\Psi_i\to\Psi_j$ for arrows $a:i\to j$ in $Q$.
We will assume that the categories $\Psi_i$ are Grothendieck \cite[\S8.3]{kashiwara_categories} (in particular they have enough injective objects, admit small inductive and projective limits, and their small filtrant inductive limits and direct sums are exact), hence we don't assume $\Psi_i$ to be small anymore.
\begin{notes}
By the definition of a Grothendieck category \cite[8.3.24]{kashiwara_categories} its small filtrant inductive limits are exact. One can represent a direct sum as a filtrant inductive limit of finite direct sums, which are exact.
\end{notes}
Note that any small abelian category \cA can be canonically 
embedded into the category $\Ind(\cA)$ of ind-objects which is a Grothendieck category \cite[\S8.6]{kashiwara_categories}.
This category is equivalent to the Quillen abelian envelope of \cA
(usually applied to exact categories \cite[Appendix A]{keller_chain}),
consisting of additive left-exact functors $F:\cA^\op\to\Mod\bZ$
\cite[\S8.6]{kashiwara_categories}.
For example, if $X$ is a noetherian scheme and $\cA=\Coh X$ is the category of coherent sheaves on~$X$, then $\Ind(\cA)$ is equivalent to the category of quasi-coherent sheaves on $X$ \cite[Appendix]{hartshorne_residues}. 
On the level of derived categories there is a natural equivalence $D^b(\cA)\iso D^b_\cA(\Ind(\cA))$ \cite[\S15.3]{kashiwara_categories}, hence the \Ext-groups are unchanged.

\begin{theorem}
\label{th:injectives}
Let $\Phi$ be a (left exact) $Q^\op$-diagram of Grothendieck categories such that $\Phi_a$ preserve products for all arrows $a$ in $Q$.
Then the category $\dRep(Q,\Phi)$ has enough injectives.
\end{theorem}
\begin{proof}
Using the standard coresolution (Theorem \ref{th:costand}),
we can embed $X\in \dRep(Q,\Phi)$ into
$\bop_{i}\sir {X_i}$.
Therefore we have to show that every $\sir{X_i}$ can be embedded into an injective object.
Let $X_i\emb I$ be an embedding into an injective object in $\Phi_i$.
Then the induced map $\sir {X_i}\to\sir I$ is a monomorphism as the functor $\si_*$ is right adjoint, hence left exact.
The object $\sir I$ is injective by Corollary \ref{cor:inj and proj}.
\end{proof}

\begin{remark}
\label{rem:both exact 1}
We will often encounter the following situation.
Let $L:\cA\to\cB$ be a functor left adjoint to a functor $R:\cB\to\cA$.
If $\cA,\cB$ are abelian categories with enough injectives and $L,R$ are exact functors, then
$$\Ext^k_\cB(LX,Y)\iso\Ext^k_\cA(X,RY)\qquad \forall X\in\cA,\,Y\in\cB,\,k\ge0.$$
Indeed, the functor $R$ maps injective objects to injective objects because of the exactness of~$L$.
As $R$ is itself exact, it maps an injective resolution of $Y$ to an injective resolution of~$RY$.
Applying the functors $\Hom(LX,-)$ and $\Hom(X,-)$ to these resolutions, we obtain the above isomorphism.
\end{remark}

\begin{theorem}
\label{th:long Psi}
Let $\Phi$ be a $Q^\op$-diagram of Grothendieck categories that admits an exact left adjoint diagram $\Psi:\cP(Q)\to\Cat$.
Then for any two objects $X,Y\in\cR=\dRep(\Phi)$,
there is a long exact sequence
\begin{multline*}
0
\to\Hom_\cR(X,Y)\to\bop_i\Hom_{\Phi_i}(X_i,Y_i)\to\bop_{a:i\to j}\Hom_{\Phi_j}(\Psi_aX_i,Y_j)\\
\to\Ext^1_\cR(X,Y)\to\bop_i\Ext^1_{\Phi_i}(X_i,Y_i)\to\bop_{a:i\to j}\Ext^1_{\Phi_j}(\Psi_aX_i,Y_j)
\to\dots
\end{multline*}
\end{theorem}
\begin{proof}
Applying the functor $\Hom(-,Y)$ to the standard resolution of $X$ (see Theorem \ref{standard})
$$0\to\bop_{a:i\to j}\si_!(\Psi_aX_i)\to\bop_{i}\si_! (X_i)\to X\to0$$
we obtain a long exact sequence.
The statement of the theorem will follow from
$$\Ext^k_\cR(\si_!{X_i},Y)\iso\Ext^k_{\Psi_i}(X_i,\si^*Y)$$
for any $X_i\in\Psi_i$ and $Y\in\cR$.
According to Remark \ref{rem:both exact 1} it is enough to show that both $\si^*,\si_!$ are exact.
The functor $\si^*$ is exact as a forgetful functor.
The functor $\si_!:\Psi_i\to\cR$ is exact because of its construction (see Theorem \ref{th:adj}),
as \Psi is exact and coproducts preserve exactness.
\end{proof}

\begin{remark}
Our next result relies on the exactness of products which appear in the construction of the functor $\si_*:\Phi_i\to\dRep(\Phi)$ (see Theorem \ref{th:right adj}).
It is known that products are not exact in Grothendieck categories in general (see e.g.\ \cite[Example 4.9]{krause_stable}).
To make our argument work we will assume that our quiver is acyclic, so that the path category $\cP(Q)$ is finite.
Then all products appearing in the construction of $\si_*$ are finite, hence exact.
Note also that the requirements of Theorem \ref{th:right adj} and Theorem \ref{th:costand} on the preservation of products by $\Phi_a$ can be omitted in this case.
\end{remark}

\begin{theorem}
\label{th:long Phi}
Let $Q$ be an acyclic quiver and $\Phi$ be an exact $Q^\op$-diagram of Grothendieck categories.
Then for any two objects $X,Y\in\cR=\dRep(\Phi)$,
there is a long exact sequence
\begin{multline*}
0
\to\Hom_\cR(X,Y)\to\bop_i\Hom_{\Phi_i}(X_i,Y_i)\to\bop_{a:i\to j}\Hom_{\Phi_i}(X_i,\Phi_aY_j)\\
\to\Ext^1_\cR(X,Y)\to\bop_i\Ext^1_{\Phi_i}(X_i,Y_i)\to\bop_{a:i\to j}\Ext^1_{\Phi_i}(X_i,\Phi_aY_j)
\to\dots
\end{multline*}
\end{theorem}
\begin{proof}
Applying the functor $\Hom_\cR(X,-)$ to the standard coresolution of $Y$ (see Theorem~\ref{th:costand})
$$0\to Y\to
\bop_{i}\sir {Y_i}\to
\bop_{a:i\to j}\sir{\Phi_aY_j}
\to0$$
we obtain a long exact sequence.
The statement of the theorem will follow from
$$\Ext^k_{\Phi_i}(\si^*X,Y_i)\iso \Ext^k_\cR(X,\sir {Y_i})$$
for any $X\in\cR$ and $Y_i\in\Phi_i$.
According to Remark \ref{rem:both exact 1} it is enough to show that both $\si^*,\si_*$ are exact.
The functor $\si^*$ is exact as a forgetful functor.
The functor $\si_*:\Phi_i\to\cR$ is exact because of its construction (see Theorem \ref{th:right adj}),
as \Phi is exact and finite products preserve exactness.
\end{proof}

\begin{remark}
\label{rem:both exact}
According to Remark \ref{rem:both exact 1},
we have $\Ext^k_{\Phi_j}\rbr{\Psi_a X_i,Y_j}
\iso\Ext^k_{\Phi_i}\rbr{X_i,\Phi_a Y_j}$
if both $\Phi$ and $\Psi$ are exact diagrams.
But this is not true in general.
\end{remark}
\section{Applications}
\label{sec:app}

\subsection{Parabolic vector bundles}
Let $X$ be a projective curve over an algebraically closed  field \bk and $w:X\to\bN$ be a map such that the set
$S=\sets{p\in X}{w_p>1}$ is finite.
A \mbox{(quasi-)} parabolic vector bundle over $X$ of type $w$ consists of a vector bundle $E$ together filtrations of the fibers
$$E_p=E_{p,0}\sps E_{p,1}\sps\dots\sps E_{p,w_p}=0\qquad \forall p\in S.$$
We define $E_p^i=E_p/E_{p,i}$.
A morphism between parabolic vector bundles $\bfE=(E,E_*)$ and $\bfF=(F,F_*)$ is a morphism $f:E\to F$ that preserves filtrations.
The category of parabolic vector bundles can be embedded into an abelian category of parabolic coherent sheaves, which is hereditary \cite{geigle_class,yokogawa_infinitesimal,heinloth_coherent,lin_spherical}.

\begin{theorem}
Given two parabolic vector bundles $\bfE=(E,E_*)$ and $\bfF=(F,F_*)$, there is an exact sequence
\begin{multline*}
0\to\Hom(\bfE,\bfF)
\to\Hom(E,F)\oplus\bop_{1\le i<w_p}\Hom(E_p^i,F_p^i)\to
\bop_{1\le i<w_p}\Hom(E_p^{i+1},F_p^i)\\
\to\Ext^1(\bfE,\bfF)
\to\Ext^1(E,F)\to 0
\end{multline*}
\end{theorem}
\begin{proof}
We can interpret parabolic vector bundles as quiver representations as follows.
For simplicity let us assume that $S=\set p$ and let $n=w_p-1$.
Let $\fm_p\sbs\cO_X$ be the maximal ideal of the point $p\in X$ and let $\bk_p=\cO_X/\fm_p$ be the corresponding skyscraper sheaf.
For any coherent sheaf $E\in\Coh X$, define its fiber
$E_p=E/\fm_pE\iso E\ts_{\cO_X}\bk_p$.
For any $V\in\Vect$ we have
$$\Hom_{\Vect}(V,(E_p)^*)
\iso \Hom_{\Vect}(E_p,V^*)
\iso\Hom_{\Coh X}(E,V^*\ts \bk_p).
$$ 
Therefore the functor 
$$\Psi_a:\Coh X\to\Vect^\op,\qquad
E\mto (E_p)^*\iso\Hom_{\Coh X}(E,\bk_p)$$
is left adjoint to the exact functor $\Phi_a:\Vect^\op\to\Coh X$, $V\mto V^*\ts \bk_p$.
Consider the quiver~$Q$
$$0\xto{a}1\xto{a_1}2\to\dots\to n$$
and the $Q$-diagram
with $\Psi_0=\Coh X$, $\Psi_i=\Vect^\op$ for $i\ge1$,
$\Psi_a$ defined as above and $\Psi_{a_i}=\Id$ for $i\ge1$.
It has a right adjoint $Q^\op$-diagram $\Phi$ with $\Phi_a$ defined as above.
A representation is given by $E\in\Coh X$, $V_i\in\Vect$ for $i\ge1$ and a chain of morphisms in \Vect
$$(E_p)^*=V_0\lto V_1\lto V_2\lto\dots\lto V_n.$$
Given a parabolic vector bundle $\bfE=(E,E_*)$, we have a chain of epimorphisms $E_p=E_p^{n+1}\to E_p^n\to\dots\to E_p^1\to E_p^0=0$ which induces a chain of monomorphisms
$$(E_p)^*=V_0\lto V_1\lto V_2\lto\dots\lto V_n,
\qquad V_i=(E_p^{n+1-i})^*.$$
We consider this data as an object of $\Rep(\Psi)$.
In this way we embed the category of parabolic vector bundles into the abelian category $\Rep(\Psi)\iso\dRep(\Phi)$.

Given two representations $\bfE=(E,V_*)$ and $\bfF=(F,W_*)$ with locally free $E\in\Coh X$, we obtain from 
Theorem \ref{th:long Phi} an exact sequence
\begin{multline*}
0\to\Hom(\bfE,\bfF)\to\Hom(E,F)\oplus\bop_{i=1}^n\Hom(W_i,V_i)\to
\bop_{i=1}^{n}\Hom(W_{i},V_{i-1})\\
\to\Ext^1(\bfE,\bfF)
\to\Ext^1(E,F)\to
\Ext^1(E,\Phi_aW_1)=0
\end{multline*}
The last equality follows from Serre duality
$\Ext^1(E,\bk_p)\iso\Hom(\bk_p,E)^*$ and the assumption that $E$ is locally free.
If $\bfE$ and $\bfF$ correspond to parabolic bundles $(E,E_*)$ and $(F,F_*)$ respectively, we have
$\Hom(W_i,V_j)\iso\Hom(E_p^{n+1-j},F_p^{n+1-i})$, hence the statement of the theorem.
\end{proof}

\begin{notes}
In \cite[p.9]{garcia-prada_betti}, given two parabolic bundles $\bfE,\bfF$, one considers $M_p\sbs\Hom(E_p,F_p)$ consisting of parabolic maps and an exact sequence of sheaves
$$0\to\lHom(\bfE,\bfF)\to\lHom(E,F)\to\Hom(E_p,F_p)/M_p\to0$$
This suggests that there there is an exact sequence
$$0\to\bop_{i=1}^n\Hom(W_i,V_i)
\to\bop_{i=1}^n\Hom(W_i,V_{i-1})
\to\Hom(E_p,F_p)/M_p\to0
$$
\end{notes}

\subsection{Nested sheaves}
Let $X$ be an algebraic surface and $p\in X$ be a point with the maximal ideal $\fm_p\sbs\cO_X$ and the skyscraper sheaf $\bk_p=\cO_X/\fm_p$.
Consider the category of pairs $(F,F')$ of coherent sheaves such that $\fm_p F\sbs F'\sbs F$.
The moduli spaces of such pairs are studied in \cite[\S4]{nakajima_perverse} under the name of moduli spaces of perverse coherent sheaves on a blow-up.
Such pairs are determined by $F$ together with an epimorphism $F_p=F/\fm_pF\onto F/F'$.

We can represent such pairs as quiver representations as follows.
Consider the quiver $Q=[0\xto a1]$ and the $Q$-diagram with $\Psi_0=\Coh X$, $\Psi_1=\Vect^\op$ and 
\eq{\Psi_a:\Coh X\to\Vect^\op,\qquad
F\mto (F_p)^*\iso\Hom_{\Coh X}(F,\bk_p).}
It has an exact right adjoint $Q^\op$-diagram $\Phi$ with $\Phi_a:\Vect^\op\to\Coh X$, $V\mto V^*\ts \bk_p$.
A representation is given by a triple $(F,V,s)$, where $F\in\Coh X$, $V\in\Vect$ and 
$s\in
\Hom_{\Vect}(V,(F_p)^*)\iso\Hom_{\Coh X}(F,V^*\ts\bk_p)$.
Given a pair $(F,F')$ as above, we consider $V^*=F/F'$
and $s$ corresponding to $F\to F/F'$.
In this way we embed the category of pairs $(F,F')$ as above into $\Rep(\Psi)$.
Given two representations $\bfE=(E,V,s)$ and $\bfF=(F,W,s')$ with a torsion free sheaf $E$ we obtain from Theorem \ref{th:long Phi}, an exact sequence 
\begin{multline}
0\to\Hom(\bfE,\bfF)\to\Hom(E,F)\oplus\Hom(W,V)\to
(E_p)^*\ts W^*\\
\to\Ext^1(\bfE,\bfF)
\to\Ext^1(E,F)\to
\Ext^1(E,\bk_p)\ts W^*
\to\Ext^2(\bfE,\bfF)\to 
\Ext^2(E,F)\to0,
\end{multline}
where we used the fact that $\Ext^2(E,\bk_p)\iso\Hom(\bk_p,E)^*=0$.

\subsection{Framed sheaves}
\label{sec:framed}
Let $X$ be an algebraic variety and let $P\in\Coh X$.
The mapping cylinder of the left exact functor $\Phi_a:\Coh X\to\Vect$, $E\mto\Hom(P,E)$, is a category with object consisting of triples $(E,V,s)$, where $E\in\Coh X$, $V\in\Vect$ and $s:V\to\Hom(P,E)$.
It can be interpreted as the category of representations of the quiver $Q=[0\xto a1]$ and the $Q^\op$-diagram $\Phi_0=\Vect$, $\Phi_1=\Coh X$ and $\Phi_a:\Phi_1\to\Phi_0$ defined as above.
This diagram has an exact left adjoint $Q$-diagram $\Psi$ with $\Psi_a:\Vect\to\Coh X$, $V\mto V\ts P$.
Given two representations $\bfE=(E,V,s)$ and $\bfF=(F,W,s')$,
we obtain from Theorem \ref{th:long Psi}
a long exact sequence
\begin{multline}
0
\to\Hom(\bfE,\bfF)\to\Hom(E,F)\oplus\Hom(V,W)
\to\Hom(V\ts P,F)\\
\to\Ext^1(\bfE,\bfF)\to\Ext^1(E,F)\to
\Ext^1(V\ts P,F)\to\dots
\end{multline}

\begin{remark}
More generally, consider a (small) abelian category $\cA$ and a left exact functor $\Phi_a:\cA\to\Vect$.
Its mapping cylinder can be identified with the category of quiver representations in the same way as above, where
a representation $\bfE=(E,V,s)$ consist of $E\in\cA$, $V\in\Vect$ and $s:V\to\Phi_aE$.
If $\Phi_a$ is an exact functor, we obtain for two representations $\bfE=(E,V,s)$ and $\bfF=(F,W,s')$ an exact sequence from Theorem \ref{th:long Phi}
\begin{multline}
0
\to\Hom(\bfE,\bfF)\to\Hom(E,F)\oplus\Hom(V,W)
\to\Hom(V,\Phi_aF)\\
\to\Ext^1(\bfE,\bfF)\to\Ext^1(E,F)\to
0
\end{multline}
and $\Ext^i(\bfE,\bfF)\iso\Ext^i(E,F)$ for $i\ge2$.
This means that the homological dimension of the mapping cylinder coincides with the homological dimension of $\cA$.
\end{remark}

\providecommand{\bysame}{\leavevmode\hbox to3em{\hrulefill}\thinspace}
\providecommand{\href}[2]{#2}


\end{document}